\documentclass[11pt]{amsart}

\usepackage{amsmath,amssymb}
\usepackage[all,curve,frame]{xy}

\usepackage[utf8]{inputenc}
\usepackage{amsfonts}
\usepackage{amsmath}
\usepackage{amsthm}
\usepackage{dsfont}
\usepackage[dvips]{graphicx}
\usepackage{enumerate}
\usepackage[all]{xy}
\usepackage{pstricks}

\theoremstyle{plain} \newtheorem{teo}{Theorem}[section]
\theoremstyle{definition} \newtheorem{defi}[teo]{Definition}

\theoremstyle{plain} \newtheorem{prop}[teo]{Proposition}
\theoremstyle{plain} \newtheorem{lema}[teo]{Lemma}
\theoremstyle{plain} \newtheorem{cor}[teo]{Corollary}
\theoremstyle{remark} \newtheorem{obs}[teo]{Remark}
\theoremstyle{remark} \newtheorem{ej}[teo]{Example}

\newcommand{\R}{\mathcal{R}}
\newcommand{\E}{\mathcal{E}}
\newcommand{\Ex}{\mbox{Ext}^1}
\newcommand{\F}{\mbox{Fac}}
\newcommand{\Hom}{\mbox{Hom}}
\newcommand{\mo}{\mbox{mod}}
\newcommand{\ad}{\mbox{add}}
\newcommand{\sti}{\mbox{s}\tau -\mbox{tilt}}

\begin{document}

\title[$\tau$-tilting modules over one-point extensions by a projective module]
{$\tau$-tilting modules over one-point extensions by a projective module}

\author[Suarez Pamela ]{Pamela Suarez}
\address{ Departamento de Matem\'atica, Facultad de Ciencias Exactas y
Naturales, Funes 3350, Universidad Nacional de Mar del Plata, 7600 Mar del
Plata, Argentina}
\email{pamelaysuarez@gmail.com}

\thanks{The author thankfully acknowledge partial support from CONICET
and from Universidad Nacional de Mar del Plata, Argentina. The results of this article are part of the PhD thesis of the author under the supervision of Sonia Trepode and Claudia Chaio. She  is grateful to them for their constant support and helpful
discussions.}

\maketitle

\begin{abstract}
Let $A$ be the one point extension of  an algebra $B$ by a projective $B$-module.  We prove that the extension of a given  support $\tau$-tilting $B$-module  is a support $\tau$-tilting $A$-module; and, conversely, the restriction of a given  support $\tau$-tilting $A$-module  is a support $\tau$-tilting $B$-module. Moreover, we prove that there exists a full embedding of quivers between the corresponding poset of support $\tau$-tilting modules.
\end{abstract}

\section*{Introduction}
Tilting theory plays an important role in Representation Theory of finite dimensional algebras. In particular,  the tilting modules were introduced in the early eighties, see for example  \cite{B, BB, HR}. The mutation process is an essential concept in tilting theory. The basic idea of a mutation is to replace an indecomposable direct summand of a tilting module by another indecomposable module in order to obtain a new tilting module. In that sense, any almost complete tilting module is a direct summand of at most two tilting modules, but it is not always exactly two. The mutation process is possible only when we have two complements. This suggests  to consider a larger class of objects. In \cite{AIR}, T. Adachi, O. Iyama and I. Reiten introduced  a class of modules called support $\tau$-tilting modules, which contains the classical tilting modules, see (\ref{definicion support}). Furthermore, the almost complete support $\tau$-tilting modules have the desired property concerning complements, that is, they have exactly two complements. A motivation to define support $\tau$-tilting modules came from Cluster Tilting Theory, since the mutation there is always possible to do. Moreover, in  \cite[Theorem 4.1]{AIR} the authors showed that there is a deep connection between $\tau$-tilting theory and cluster-tilting theory. They also showed that the notion of support $\tau$-tilting modules is connected with silting theory, see \cite[Theorem 3.2]{AIR}.

Since $\tau$-tilting theory is a generalization of tilting theory, many properties of tilting modules are preserved by support $\tau$-tilting modules. In \cite{AHT}, for one point extension algebras I. Assem, D. Happel and S. Trepode studied how to extend and restrict tilting modules. More precisely, if $A=B[P_0]$ is the one point extension of $B$ by the projective $B$-module $P_0$, they showed how to construct in a natural way a tilting $A$-module from a tilting $B$-module and conversely, given a tilting $B$-module they constructed a tilting $A$-module. Motivated by this fact, in this article we shall study the behaviour of support $\tau$-tilting modules for one point extension. Consider $\R=_BU_A\otimes _{A} \--$ from $ \mo \,A$ to $ \mo \,B$ and $\E=\Hom_B(_BU_{A}, \--)$ from $ \mo \,B $ to $\mo \,A $ the restriction and the extension functors, respectively. We prove the following result:

\vspace{0.1in}
\noindent \textbf{Theorem A.}
\textit{Let $B$ be a finite dimensional $k$-algebra over an algebraically closed field. Let $A=B[P_0]$ the one point extension of $B$ by the projective $B$-module $P_0$ and $S$ the simple $A$-module corresponding to the new vertex of $A$. Then,
\begin{enumerate}[(a)]
\item If  $M$ is a basic support $\tau$-tilting $B$-module then   $\E M \oplus S$ is a support $\tau$-tilting $A$-module.
\item If $T$ is a basic support $\tau$-tilting $A$-module then  $\R T $ is a  support $\tau$-tilting $B$-module.
\end{enumerate}}
\vspace{0.1in}

As a direct consequence, we obtain that the functors $\R$ and $\E$ induce morphisms $r$ from $\sti\,A$ to $ \sti\,B$ and $e$ from $\sti\,B$ to $\sti\,A$ such that $re=\mbox{id}_{\sti\,B}$, where $\sti\,B$ ($\sti\,A$, respectively) is the set of isomorphism classes of basic support $\tau$-tilting modules over $B$ ($A$, respectively).

In \cite[Proposition 6.1]{AHT} the authors proved that if $B$ is a hereditary algebra, $A=B[P_0]$ and $T$ a tilting $B$-module then $\mbox{End}_A e T$ is a one-point extension of $\mbox{End}_B T$. In this work,  we generalize the same result for any algebra $B$, $A=B[P_0]$ and $T$ a $\tau$-tilting $B$-module. On the other hand, in \cite[Theorem 5.2]{AHT}, the authors also showed that there exists a full embedding of quivers between the poset of tilting modules. We prove that the above mentioned result still holds true for support $\tau$-tilting modules, as we state in the next theorem.

\vspace{0.1in}
\noindent\textbf{Theorem B.}
\textit{Let $B$ a finite dimensional $k$-algebra over an algebraically closed field and $A=B[P_0]$ be the one point extension of $B$ by the projective $B$-module $P_0$. Then the map $e:\sti\,B\rightarrow \sti\,A$ induces a full embedding of quivers $e:Q(\sti\,B)\rightarrow Q(\sti\,A)$.}
\vspace{0.1in}

Finally, we point out some technical properties concerning the successors and the predecessors of a support $\tau$-tilting module which belong to the image of $e$.

We observe that most of the statements  fail if we drop the assumption that the module $P_0$ is projective.

This paper is organized as follows. In the first section, we present some notations and preliminaries results. Section 2 is dedicated to prove Theorem A and the results concerning  the relationship between the support $\tau$-tilting $B$-modules and the support $\tau$-tilting $A$-modules. We study their torsion pairs and their endomorphism algebras. In Section 3, we prove Theorem B and state some technical consequences.

\section{Preliminaries}

Throughout this paper, all algebras are basic connected finite dimensional algebras over an algebraically closed field $k$.
\subsection{}
For an algebra $A$ we denote by $\mo\,A$ the category of finitely generated left $A$-modules. An algebra $B$ is called a \emph{full subcategory} of $A$ if there exists an idempotent $e \in A$ such that $B=eAe$. An algebra $B$ is called \emph{convex} in $A$ if, whenever there exists a sequence $e_i=e_{i_0}, e_{i_1}, \cdots e_{i_t}=e_j$ of primitive orthogonal idempotents such that $e_{i_{l+1}}Ae_{e_{i_l}}\neq 0$ for $0\leq l < t$, $ee_i=e_i$ and $ee_j=e_j$, then $ee_{i_l}=e_{i_l}$, for each $l$.

For a subcategory $C$ of $\mo\,A$, we define
\[\mathcal{C}^{\bot}=\{X\in \mo\,A \, | \,\Hom_A(\mathcal{C},X)=0\}\]
and,
\[\mathcal{C}^{{\bot}_1}=\{X\in \mo\,A \,|\, \Ex_A(\mathcal{C},X)=0\}.\]

Dually, the categories ${}^{\bot}\mathcal{C}$ and ${}^{{\bot}_1}\mathcal{C}$ are defined.

Recall that a subcategory $\mathcal{X}$ of an additive category $\mathcal{C}$ is said to be \emph{contravariantly finite} in $\mathcal{C}$ if for every object $M$ in $\mathcal{C}$ there exist some $X\in\mathcal{X}$ and a morphism $f:X\rightarrow M$ such that for every $X'\in \mathcal{X}$ the sequence $\Hom_{\mathcal{C}}(X',X)\stackrel{f}\rightarrow \Hom_{\mathcal{C}}(X',M) \rightarrow 0$ is exact. Dually we define \emph{covariantly finite subcategories} in $\mathcal{C}$. Furthermore, a subcategory of $\mathcal{C}$ is said to be \emph{functorially finite} in $\mathcal{C}$ if it is both contravariantly an covariantly finite in $\mathcal{C}$.

We denote by $D$ the usual standard duality $\Hom_k(\--,k):\mo\,A\rightarrow \mo\,A^{op}$, see \cite[I, 2.9]{ASS}.

For an $A$-module $X$, we denote by $\ad\,X $ the full subcategory of $\mo\,A$ whose objects are the direct sums of direct summands of $X$, and by $\F X$  the full subcategory of $\mo\,A$ whose objects are the factor modules of finite direct sums of copies of $X$.

Finally, we say that an $A$-module $X$ is \textit{basic} if the indecomposable direct summands of $X$ are pairwise non-isomorphic.

\subsection{}
Let $B$ be an algebra and $P_0$ be a fixed projective $B$-module. We denote by $A=B[P_0]$ the one point extension of $B$ by $P_0$, which is, the matrix algebra
\[
A=\left(
\begin{array}{ll}
B & P_0 \\
0 & k
\end{array}
\right)
\]
with the ordinary matrix addition and the multiplication induced by the module structure of $P_0$.

It is well known that $B$ is a full convex subcategory of $A$, and that there is a unique indecomposable projective $A$-module $\widetilde {P}$ which is not a projective $B$-module. Moreover, the simple top  $S$ of $\widetilde{P}$ is an injective $A$-module and $\mbox{pd}_AS \leq 1$, where by $\mbox{pd}_A S $ we mean the projective dimension of the simple $S$.

The right perpendicular category of $S$ is the full subcategory of $\mo\,A$ defined as follows
\[S^{\tiny{\mbox{perp}}}=\{X\in \mo\,A \,|\, \Hom_A(S,X)=0 \hspace{0.1in} \mbox{and} \hspace{0.1in} \Ex_A(S,X)=0\}.\]

Consider the $B$-$A$-bimodule $U=e_B \,A$, where  $e_B$ is the identity of $B$. The module $U$ is  projective as a right $A$-module and as a left $B$-module.

We consider the restriction functor and the extension functor define as follows, respectively:
\[\R=_BU_A\otimes _{A} \--: \mo \,A \rightarrow \mo \,B\]
\[\E=\Hom_B(_BU_{A}, \--): \mo \,B \rightarrow \mo \,A  \]

It is well known that $(\R,\E)$ is an adjoint pair. Moreover, since $U$ is projective, the functors $\R$ and $\E$ are exact. The unit $\delta: \mbox{id}_{\tiny{\mo\,A}}\rightarrow \E\R$  and the co-unit $\epsilon:\R\E\rightarrow \mbox{id}_{\tiny{\mo\,B}}$  are defined as follows: for $X$  an $A$-module, the functor
\begin{eqnarray*}
\delta_X: & X&\rightarrow  \Hom_B(U,U\otimes_A X)
\end{eqnarray*}
is defined by sending $x$ to $ u\rightarrow u\otimes x$, for $x \in X$ and $u \in U$ . If $M$ is a $B$-module, then
\begin{eqnarray*}
\epsilon_M:&U\otimes_A \Hom_A(U,M)&\rightarrow M
\end{eqnarray*}
 is such that sends $u\otimes f $ to$ f(u)$, for $u \in U$ and $f \in \Hom_B(U,M)$.

If we consider $\mo\,B$  embedded in $\mo\,A$ under the usual embedding functor, then $\R X$  is a submodule of $X$.

The functor $\E$ is full and faithfull. In particular, it preserves indecomposability. The functor $\R$ is the torsion radical of the torsion pair $(\mo\,B,\ad \,S)$ in $\mo\,A$, and the canonical sequence of an $A$-module $X$ in this torsion pair is the exact sequence
\[0 \rightarrow \R X \rightarrow X \rightarrow S^{r_{_X}}\rightarrow 0 \]
with $r_{_X}=\mbox{dim}_k\Hom_A(X,S)$.

The next proposition lists relevant properties of these functors, stated in \cite[Proposition 2.5, Lemma 3.1]{AHT}

\begin{prop}\label{funcprop}
The adjoint pair of functors $(\R,\E)$satisfies the following properties:
\begin{enumerate}
\item[(a)] The co-unit $\epsilon$ is a functorial isomorphism.
\item[(b)] Let $X$ be an $A$-module. $\Hom_A(S,X)=0$ if and only if $S$ is not a direct summand of $X$.
\item[(c)] If $X\in S^{\tiny{\mbox{perp}}}$, then $\delta_X:X\rightarrow \E\R X$ is a functorial isomorphism.
\item[(d)] Let $M$ be a $B$-module, then $ \E M \in S^{\tiny{\mbox{perp}}}$.
\end{enumerate}
\end{prop}

As a consequence of the above proposition, we have that $\mo\,B$ and $S^{\tiny{\mbox{perp}}}$ are equivalent categories.

Next, we state the following result, which will be very useful in this paper.

\begin{prop}\cite[Corollary 3.4, Proposition 3.6]{AHT}\label{teoext}
Let $X,Y$ be $A$-modules and $M$ be a $B$-module. Then the following conditions hold.
\begin{enumerate}[(a)]
\item $\Ex_A(X,\E M) \cong \Ex_B(\R X, M)$.
\item If $X\in S^{\tiny{\mbox{perp}}}$, then $\Ex_A(\E M, X) \cong \Ex_B(M,\R X)$.
\item  There is an epimorphism $\Ex_A(X,Y)\rightarrow \Ex_A(\R X, \R Y)$.  Moreover, if $Y \in S^{\tiny{\mbox{perp}}}$ then the morphism is an isomorphism.
\end{enumerate}

\end{prop}

\subsection{}

We recall some  results on $\tau$-tilting modules. For a detail account on $\tau$-tilting theory we refer the reader to \cite{AIR}.

\begin{defi}\cite[Definition 0.1]{AIR}\label{definicion support}
Let $A$ be a finite dimensional algebra.
\begin{enumerate}[(a)]
\item An $A$-module $M$ is $\tau$-rigid if $\Hom_A(M,\tau M)=0$.
\item An $A$-module $M$ is $\tau$-tilting (almost complete $\tau$-tilting, respectively) if $M$ is $\tau$-rigid and $|M|=|A|$ ($|M|=|A|-1$, respectively).
\item An $A$-module $M$ is support $\tau$-tilting if there exists an idempotent $e$ of $A$ such that $M$ is a $\tau$-tilting $A/ \langle e \rangle $-module.
\end{enumerate}
\end{defi}

For the convenience of the reader we state \cite[Proposition 5.8]{AS} which will be useful for our further purposes.

\begin{prop}\label{propfac}
Let $X,Y \in \mo\,A$. The following conditions hold.
\begin{enumerate}[(a)]
\item $\Hom_A(X,\tau Y)=0$ if and only if $\Ex_A(M, \F\, N)=0$.
\item $M$ is $\tau$-rigid if and only if $M$ is $\mbox{Ext}$-projective in $\F \,M$.
\end{enumerate}
\end{prop}

The next result gives a relationship between the torsion classes and the support $\tau$ -tilting modules. We denote by $\sti\,A$ the set of isomorphism classes of basic support $\tau$-tilting $A$-modules and by $f-\mbox{tors}\,A$ the set of functorially finite torsion classes in $\mo\,A$.

\begin{teo}\cite[Theorem 2.7]{AIR}
There is a bijection between $f-\mbox{tors}\,A$ and $\sti\,A$ given by $\mathcal{T} \rightarrow P(\mathcal{T})$ with inverse $M\rightarrow \F \,M$.
\end{teo}

\begin{obs}
Note that the inclusion in $f-\mbox{tors}\,A$ gives rise to a partial order on $\sti\,A$, as follows: \textquotedblleft $ U \leq T$ if and only if $\F\,U \subset \F\,T$". Then, $\sti\,A$ is a partially ordered set.
\end{obs}

For $\tau$-tilting modules, we have a result which is an analog to Bongartz Lemma for tilting modules. For the convenience of the reader we state Bongartz's Lemma below.

\begin{teo}\cite[Theorem 2.10]{AIR}
Let $U$ be a $\tau$-rigid $A$-module. Then, $\mathcal{T}= {}^{\bot}(\tau U)$ is a sincere functorially finite torsion class and $T=P(\mathcal{T})$ is a $\tau$-tilting $A$-module satisfying $U\in\ad\,T$ and $ {}^{\bot}(\tau U)=\F \,T$.
\end{teo}

$P( {}^{\bot}(\tau U))$ is said to be the \emph{Bongartz completion}  of $U$.

\vspace{0.2in}
We have the following characterizations for a $\tau$-rigid module to be a $\tau$-tilting module.

\begin{teo}\cite[Theorem 2.12]{AIR}\label{teotauequiv}
The following conditions are equivalent for a $\tau$-rigid module $T$.
\begin{enumerate}[(a)]
\item $T$ is $\tau$-tilting.
\item $T$ is maximal $\tau$-rigid, i.e., if $T\oplus X$ is $\tau$-rigid for some $A$-module $X$, then $X\in \ad\,T$.
\item ${}^{\bot}(\tau T)=\F \, T$.
\end{enumerate}
\end{teo}

It is  convenient to see the support $\tau$-tilting $A$-modules and the $\tau$-rigid $A$-modules, as certain pair of $A$-modules. More precisely,

\begin{defi}\cite[Definition 0.3]{AIR}
Let $(M,P)$ be a pair with $M\in \mo\,A$ and $P$ a projective $A$-module.
\begin{enumerate}[(a)]
\item  If $M$ is $\tau$-rigid and $\Hom_A(P,M)=0$ then $(M,P)$ is a $\tau$-rigid pair.
\item  If $(M,P)$ is $\tau$-rigid and $|M|+|P|=|A|$ ($|M|+|P|=|A|-1$, respectively) then $(M,P)$ is a support $\tau$-tilting (almost complete support $\tau$-tilting, respectively) pair.
\end{enumerate}
\end{defi}

It follows from \cite[Proposition 2.3]{AIR}, that  the notion of support $\tau$ -tilting module and the notion of support $\tau$ -tilting pair are essentially the same.

We say that $(X,0)$ ($(0,X)$, respectively) with $X$ an indecomposable module is a complement of an almost complete support $\tau$-tilting pair $(U,Q)$ if $(U\oplus X,Q)$ ($(U,Q\oplus X)$, respectively) is a support $\tau$-tilting pair.

\begin{teo}\cite[Theorem 2.18]{AIR}
Any basic almost complete support $\tau$-tilting pair for $\mo\,A$ has exactly two complements.
\end{teo}

Two completions $(T,P)$ and $(T',P')$ of an almost complete support $\tau$-tilting pair $(U,Q)$ are called mutations one of each other. We write $(T',P')=\mu_{(X,0)}(T,P)$ ($(T',P')=\mu_{(0,X)}(T,P)$, respectively) if  $(X,0)$ ($(0,X)$, respectively) is a complement of $(U,Q)$ giving rise to $(T,P)$.

\begin{defi}\cite[Definition 2.28]{AIR}\label{leftmutation}
Let $T=X \oplus U$ and $T'$ be  support $\tau$-tilting $A$-modules such that $T'=\mu_X T$ for some indecomposable $A$-module $X$. We say that $T'$ is a left mutation ( right mutation, respectively) of $T$ and we write $T'=\mu_X^{-}T$ ( $T=\mu_X^{+}T$, respectively) if the following equivalent conditions are satisfied.
\begin{enumerate}[(a)]
\item $T>T'$ ($T<T'$, respectively).
\item $X \notin \F \,U$ ($X \in \F \,U$, respectively).
\item ${}^{\bot}(\tau U)\subseteq {}^{\bot}(\tau X)$ (${}^{\bot}(\tau U) \nsubseteq {}^{\bot}(\tau X)$, respectively).
\end{enumerate}
\end{defi}



\begin{defi}\cite[Definition 2.29]{AIR}
The support $\tau$-tilting quiver $Q(\sti\,A)$ of $A$ is defined as follows:
\begin{itemize}
\item The set of vertices consists of the isomorphisms classes of basic support $\tau$-tilting $A$-modules.
\item There is an arrow from $T$ to $U$ if $U$ is a left mutation of $T$.
\end{itemize}
\end{defi}

\begin{obs}
Note that this exchange graph is $n$-regular, where $n=|A|$ is the number of simple $A$-modules.

It follows from \cite[Corollary 2.34]{AIR} that the exchange quiver $Q(\sti\,A)$ coincides with the Hasse quiver of the partially ordered set $\sti\,A$.
\end{obs}

\section{Extension and Restriction maps}

Throughout this section, we assume that $A$ is the one-point extension of $B$ by the projective $B$-module $P_0$. We study the relationship between the support $\tau$-tilting $B$-modules and the support $\tau$-tilting $A$-modules.

We start with a remark which shall be very useful for our purposes.

\begin{obs}\label{obsimp}

Let $Y$ be an  $A$-module  such that $\Ex_A(S,Y)=0$. Then $Y=Y'\oplus S^r$ with $Y'\in S^{\tiny{\mbox{perp}}}$ and $r\geq 0$.
In fact,  first assume that $\Hom_A(S,Y)=0$. Then, by (\ref{funcprop}), $Y=Y'$ and $r=0$.
Now if  $\Hom_A(S,Y)\neq 0$, then, by (\ref{funcprop}), $S$ is a direct summand of  $Y$, namely, $Y=S\oplus Z$. Note that $\Ex_A(S,Z)=0$. If $\Hom_A(S,Z)=0$ we are done. Otherwise, $S$ is a direct summand of $Z$ and $Z=Z_1\oplus S$. Moreover,  $Y=S^2\oplus Z'$. Iterating this argument over $Z_i$, for $i=1\cdots r-1$, we get $Y=Y'\oplus S^r$.
\end{obs}

The following result shows us how to extend $\tau$-rigid $B$-modules.

\begin{teo}\label{teoextend}
Let $T$ be a   $\tau$-rigid $B$-module. Then $\E T\oplus S$ is a  $\tau$-rigid $A$-module.
\end{teo}

\begin{proof}
Consider $T$  a   $\tau$-rigid $B$-module. By Proposition \ref{propfac}, we have that $\mbox{Ext}^1_B(T,\mbox{Fac}T)=0$.
Let us show that $\Ex_A(\E T \oplus S,\F (\E T \oplus S))=0$.

Note that, in this case, $\F (\E T \oplus S) = \F (\E T)\oplus \F S$. Then,
\begin{eqnarray*}
\Ex_A(\E T\oplus S, \F (\E T \oplus S))&=&\Ex_A(\E T\oplus S, \F (\E T) \oplus \F\;S)\\
&=& \Ex_A(\E T, \F (\E T))\oplus \Ex_A(S, \F (\E T)) \oplus \Ex_A(\E T\oplus S, \F\; S)
\end{eqnarray*}
Since $\F\; S= \{0, S\}$ and $S$ is an injective module, then $\Ex_A(\E T\oplus S, \F\; S)=0$.

Now, we show that $\Ex_A(S, \F (\E T))=0$. Consider $Y \in \F (\E T)$. By definition, there exists an epimorphism $f:M \rightarrow Y$, with $M\in \ad(\E T)$. Applying $\Hom _A(S,\--)$ we have
\[\Ex_A(S,M)\rightarrow \Ex_A(S,Y)\rightarrow \mbox{Ext}^2_A(S,\mbox{Ker}f)\]
since $M \in \ad(\E T)$ and $\mbox{pd}_A S\leq 1$ then $\Ex_A(S,M)=0$  and $\mbox{Ext}^2_A(S,\mbox{Ker}f)=0$, respectively . Thus, $\Ex_A(S,Y)=0$. Then, $\Ex_A(S, \F (\E T))=0$.

Finally, we prove that $\Ex_A(\E T, \F (\E T))=0$. Let $W \in \F (\E T) $. By definition, there exists an epimorphism $g:N \rightarrow W$, with $N\in \ad(\E T)$. Applying the functor $\R$ to $g$, we get that $\R W \in \F \;T$, because $\R N \in \ad(T)$. Since $T$ is a  $\tau$-rigid $B$-module, then $\Ex_B(T,\R W)=0$.

On the other hand, since $W \in \F (\E T)$ and $\E T \in S^{\tiny{\mbox{perp}}}$, then $\Ex_A(S,W)=0$. By  Remark \ref{obsimp}, we have that $W= S^j \oplus W'$, with $W'\in S^{\tiny{\mbox{perp}}}$ and $j\geq 0$. Thus, by (\ref{teoext}),
\begin{eqnarray*}
\Ex_A(\E T,W) &=& \Ex_A(\E T,W')\oplus \Ex_A(\E T , S^j)\\
&=&\Ex_B(T,\R W')\\
&=&0.
\end{eqnarray*}
Therefore, $\Ex_A(\E T \oplus S,\F (\E T \oplus S))=0$. Moreover, by Proposition \ref{propfac},  $\E T \oplus S$ is a  $\tau$-rigid $A$-module.
\end{proof}





 Now, we consider the opposite problem, namely, given a $\tau$-rigid $A$-module $T$ we will prove that $\R T$ is a $\tau$-rigid $B$-module. In order to do that, first we establish a relationship between $\tau_B \R T$ and $\tau_A T$.

\begin{lema}\label{lemppm}
Let $P_1\rightarrow P_0 \rightarrow T \rightarrow 0$ be a minimal projective presentation of $T$ in $\mo\,A$ and  $Q_1\rightarrow Q_0 \rightarrow \R T \rightarrow 0$ be a minimal projective presentation of  $\R T$ in $\mo\,B$. Then $Q_i$ is a direct summand of $\R P_i$, for $i=0,1$. In particular, $Q_i$ is  isomorphic to a submodule of $P_i$,  for i=0,1.
\end{lema}

\begin{proof}
Consider
\begin{equation}\label{presproy}
P_1\rightarrow P_0 \rightarrow T \rightarrow 0
\end{equation}
 a minimal projective presentation of $T$ in $\mo\,A$. Applying the functor $\R$ to (\ref{presproy}), we obtain the following exact sequence
\begin{equation}\label{presproy2}
\R P_1\rightarrow \R P_0 \rightarrow \R T \rightarrow 0
\end{equation}
which is a projective presentation of $\R T$, but it  is not necessarily minimal. By the universal property of the projective cover, we have the following commutative diagram with exact rows and columns:

\begin{center}
$
    \xymatrix  @!0 @R=1cm  @C=1.5cm {
       \R P_1\ar[d]\ar[r]&\R P_0\ar[d]\ar[r] & \R T\ar[d]\ar[r] & 0     \\
       Q_1\ar[d]\ar[r] & Q_0 \ar[d]\ar[r]& \R T \ar[d]\ar[r] & 0 \\
       0 & 0 &0 &}
        $
\end{center}
Since  $Q_0$ is a  projective $B$-modules, then $Q_0$ is a direct summand of $\R P_0$, and therefore $Q_0$ is a submodule of $P_0$. Similarly, $Q_1$ is a direct summand of $\R P_1$ and a submodule of $P_1$.
\end{proof}

\begin{teo}\label{teotrasladado}
Let  $T$ be an $A$-module. Then $\tau_A \R T$ is isomorphic to a submodule of $\tau_A T$.
\end{teo}
\begin{proof}
By Lemma \ref{lemppm}, we have the following commutative diagram with exact rows and columns:

\begin{center}
$
    \xymatrix  @!0 @R=1cm  @C=1.5cm {
       0\ar[d]& 0\ar[d] & 0\ar[d] &     \\
       Q_1\ar[d]\ar[r] & Q_0 \ar[d]\ar[r]& \R T \ar[d]\ar[r] & 0 \\
       P_1 \ar[r] & P_ 0 \ar[r] & T \ar[r]& 0}
        $

\end{center}
Applying  $\Hom_A (\--,A)$ to the above diagram, we obtain the following commutative diagram with exact rows and columns:

\begin{center}
$
    \xymatrix  @!0 @R=1.2cm  @C=1.5cm {
       \Hom_A(P_0,A)\ar[d]\ar[rr]&&\Hom_A(P_1,A)\ar[d]\ar[rr] && \,\,\mbox{Tr}_AT\ar[d]\ar[r] & 0     \\
       \Hom_A(Q_0,A)\ar[d]\ar[rr] && \Hom_A(Q_1,A) \ar[d]\ar[rr]&& \mbox{Tr}_A \R T\ar[d]\ar[r] & 0 \\
       0 && 0 &&0 &}
        $

\end{center}
where the morphism $\mbox{Tr}_AT \rightarrow \mbox{Tr}_A\R T$ is the  morphism induced by passing through cokernels. Applying the duality functor to the epimorphism $\mbox{Tr}_AT \rightarrow \mbox{Tr}_A\R T$ , we get
\begin{equation*}
0\rightarrow \tau_A \R T \rightarrow \tau_A T.
\end{equation*}
Thus, $\tau_A(\R T)$ is isomorphic to a submodule of $\tau_A( T)$ proving the result.
\end{proof}

The following result is an immediate consequence of Theorem \ref{teotrasladado}.

\begin{cor}
Let $T$ be an $A$-module. Then,
\begin{enumerate}[(a)]
\item $\tau_B \R T$ is isomorphic to a submodule of $\tau_A T$.
\item $\tau_B \R T$ is isomorphic to a submodule of $\R (\tau_A T)$.
\end{enumerate}
\end{cor}

\begin{proof}

$(a).$ Since  $\R T$ is a  $B$-module, it follows from \cite[V, Ex 5]{ARS} that  $\tau_B \R T$ is isomorphic to a submodule of $\tau_A \R T $. By Theorem \ref{teotrasladado}, we have that $\tau_A \R T $ is isomorphic to a submodule of  $\tau_A T$. Therefore,  $\tau_B \R T $ is isomorphic to a submodule of $\tau_A T$.

$(b).$ By Statement $(a)$, $\tau_B \R T$ is a submodule of $\tau_A T$. Since $\R$ is an exact functor and $\tau_B \R T$ is a $B$-module, then $\tau_B \R T$ is a submodule of $\R (\tau_A T)$.

\end{proof}

The following example shows  that  $\tau_B \R T $ is in general  a proper submodule of $\R (\tau_A(T))$.

\begin{ej}
Consider $B$   the algebra given by the quiver
$ \xymatrix{
1 \ar@/_/[r]_{\beta}
& 2 \ar@/_/[l]_{\alpha}
}$
with the relation $\alpha\beta=0$.

Let $A=B[P_2]$ be  the one-point extension of $B$ by the projective $P_2$. Then $A$ is given by the quiver $ \xymatrix{
1 \ar@/_/[r]_{\beta}
& 2 \ar@/_/[l]_{\alpha}
& 3 \ar[l]_{\gamma}
}$
with  $\alpha\beta=0$.

Let $M$ be a module whose composition factors are $\small{\txt{3\\2\\1}}$. Then we have that $\R M= \small{\txt{2\\1}}$ and $\tau_B \R M= \small{\txt{1\\2}}$.

On the other hand, $\tau_A M= \small{\txt{\,\,\,\,\,\,2\\ 3 \,\,1 \\2}}$ and $\R (\tau_A M)= \small{\txt{2\\1\\2}}$. Hence, one can clearly see that  $\tau_B \R M$ is  a proper submodule of $\R (\tau_A M)$.
\end{ej}
\vspace{0.1in}
The above result will lead us to obtain that the restriction functor behaves well with $\tau$-rigid modules, as we state in the next theorem.

\begin{teo}\label{restrig}
Let $T$ be a  $\tau$-rigid $A$-module. Then $\R T$ is a $\tau$-rigid $A$-module.
\end{teo}

\begin{proof}
Consider $T$  a $\tau$-rigid $A$-module. Then $\Ex_A(T,T)=0$. By Proposition \ref{teoext},(c), we have that
$\Ex_A(\R T,\R T)=0$. Thus, $D\overline{\Hom}_A(\R T,\tau_A\R T)=0$.

Let us prove that  $\Hom_A(\R T, \tau_A \R T)=0$. Let $f: \R T \rightarrow \tau_A \R T$ be a nonzero morphism. Since $D\overline{\Hom}_A(\R T,\tau_A\R T)=0$, it follows that $f$ factors through an injective $A$-module $I$. Then  $f=hg$, with $g:\R T \rightarrow I$ and $h:I \rightarrow \tau_A \R T$. Since $\R T$ is a submodule of $T$ and $I$ is an injective $A$-module,  there exists a morphism  $g^*:T\rightarrow I$ such that $g=g^*j$, where $j:\R T \rightarrow T$ is the natural inclusion.
\begin{center}
$
    \xymatrix  @!0 @R=1cm  @C=1cm {
                &\R T \ar@{^{(}->}[dl]_j\ar[dr]^g\ar[rr]^f &           & \tau_A (\R T) \ar@{^{(}->}[rr]^i && \tau_A T     \\
       T \ar@{-->}[rr]_{g^{*}}&                           & I \ar[ur]_h &                      & & }
        $
\end{center}

Since $f \neq 0$, we have a nonzero morphism $ihg*:T\rightarrow \tau_A T$, where $i:\tau_A(\R T)\rightarrow \tau_A T$ is the inclusion morphism defined in the proof of  Theorem \ref{teotrasladado}. Therefore, $\Hom_A(T,\tau_A T)\neq 0$, contradicting the hypothesis. Hence,  $\R T$ is a  $\tau$-rigid $A$-module.
\end{proof}

As a consequence of above Theorem, we deduce the following corollary.

\begin{cor}
Let $T$ be a $\tau$-rigid $A$-module. Then $\R T $ is a $\tau$-rigid $B$-module.
\end{cor}

\begin{proof}
Assume that $\Hom_B(\R T,\tau_B \R T)\neq 0$. Since  $\tau_B \R T$ is isomorphic to a submodule of $\tau_A \R T$, then $\Hom_A(\R T, \tau_A \R T)\neq 0$, which is a contradiction to Theorem \ref{restrig}. Therefore,  $\R T$ is a  $\tau$-rigid $B$-module.
\end{proof}

The next result is the main key to prove that the restriction of a support $\tau$-tilting $A$-module is a support $\tau$-tilting $B$-module.

\begin{prop}\label{prop_ortog}
Let  $T$ be a $\tau$-rigid $A$-module and $X$ be a $B$-module. If  $X\in {}^\bot(\tau_B\R T)$ then $\E X \in {}^\bot(\tau_A T)$.
\end{prop}

\begin{proof}
Let $X\in {}^\bot(\tau_B\R T)$. Then, $\Hom_B(X,\tau_B\R T)=0$. By Proposition \ref{propfac}, we have that $\Ex_B(\R T, \F X)=0$. We shall prove  that $\Ex_A( T, \F (\E X))=0$.

Let $Y \in \F (\E X)$, then there exists an epimorphism $f: M \rightarrow Y$, with  $M \in \ad (\E X)$. Since  $\E X \in S^{\tiny{\mbox{perp}}}$, then $\Ex_A(S,Y)=0$. Thus, by Remark (\ref{obsimp}),  we have that $Y=Y' \oplus S^r$, with  $Y'\in S^{\tiny{\mbox{perp}}}$ and $r\geq0$.

Applying the functor $\R$ to the morphism $f: M \rightarrow Y'\oplus S^r$, we obtain that $\R Y' \in \F X$, and thus  $\Ex_B(\R T, \R Y')=0$. Then, by Proposition \ref{teoext},(a), $\Ex_A(T, \E\R Y')=0$. Since $Y'\in S^{\tiny{\mbox{perp}}}$, then $\Ex_A(T,Y')=0$. Therefore,
\begin{eqnarray*}
\Ex_A(T,Y)&\cong&\Ex_A(T,Y'\oplus S^r)\\
&\cong &\Ex_A(T,Y')\oplus \Ex_A(T,S^r)\\
&\cong &0
\end{eqnarray*}
because $S$ is an injective module.

Then, $\Ex_A( T, \F (\E X))=0$ and, by Proposition \ref{propfac},(a), we get the result.
\end{proof}





Now, we are in position to prove Theorem A.

\begin{teo}\label{teopares}
Let $B$  be an algebra and $A=B[P_0]$. Then,
\begin{enumerate}[(a)]
\item If  $(M,Q)$ is a basic $\tau$-rigid  (support $\tau$-tilting, respectively) pair for $\mo\,B$, then   $(\E M \oplus S, Q)$ is a  $\tau$-rigid  (support $\tau$-tilting, respectively) pair for $\mo\,A$.
\item If $(T,P)$ is a basic $\tau$-rigid  (support $\tau$-tilting, respectively) pair for $\mo\,A$, then  $(\R T , P^*)$ is a  $\tau$-rigid  (support $\tau$-tilting, respectively) pair for $\mo\,B$, where $P^*$ is the projective $B$-module which is obtained  by $P$ removing the projective $A$-module  $\widetilde{P}$.
\end{enumerate}
\end{teo}

\begin{proof}
$(a).$ Let $(M,Q)$ be a $\tau$-rigid pair for $\mo\,B$. By Theorem \ref{teoextend}, we know that $\E M \oplus S$ is a $\tau$-rigid $A$-module. On the other hand,
\begin{eqnarray*}
  \Hom_A(Q,\E M\oplus S) &\cong& \Hom_A(Q,\E M)\oplus \Hom_A(Q,S)  \\
   &\cong&\Hom_B(\R Q, M)  \\
    &\cong&\Hom_B(Q,M)\\
    &\cong&0
\end{eqnarray*}
where $\Hom_A(Q,S)=0$ because $Q$ is a $B$-module. Hence  $(\E M \oplus S, Q)$ is a  $\tau$-rigid pair for $\mo\,A$.

In addition, if   $(M,Q)$ is a support $\tau$-tilting pair, then  $|M|+|Q|=|B|$. Since $\E$ is a faithful functor, then $|M|=|\E M|$. Moreover, since  $\E M \in S^{\tiny{\mbox{perp}}}$ then $S$ is not a direct summand of  $\E M$. Hence, $|\E M \oplus S|=|\E M|+1$ and
\begin{eqnarray*}
 |\E M \oplus S|+|Q|&=&1+|\E M|+|Q|\\
 &=&1+|B|\\
 &=&|A|.
\end{eqnarray*}

$(b).$ Let $(T,P)$ be a  $\tau$-rigid pair for $\mo\,A$. By Theorem  \ref{restrig}, we have that $\R T$ is a  $\tau$-rigid $B$-module. Therefore, it is only left to prove that $\Hom_B(P^*,\R T)=0$. We know that $\R T$ is a submodule of $T$, then
\[0\rightarrow \Hom_A(P^*,\R T)\rightarrow \Hom_A(P^*,T)=0.\]
Thus , $\Hom_B(P^*,\R T)=0$. Hence, $(\R T, P^*)$ is a  $\tau$-rigid pair for $\mo\,B$.

In addition, if $(T,P)$ is a support $\tau$-tilting pair for $\mo\,A$,  we shall show that $(\R T, P^*)$ is a  support $\tau$-tilting pair for $\mo\,B$.

First, assume that $\widetilde{P}$ is a direct summand of $P$. Then $T$ is a  $B$-module, and thus $\R T \cong T$. Therefore, we have
\begin{eqnarray*}
|\R T|+|P^*|&=&|T|+|P|-1\\
&=&|A|-1\\
&=&|B|.
\end{eqnarray*}

Next, assume that $\widetilde{P}$ is not a direct summand of $P$. Then $P^*=P \cong \R P$. According to \cite[Corollary 2.13]{AIR}, we have to prove that $\F \R T= {}^\perp(\tau_B\R T)\cap P^{ \perp}$. Since  $\Hom_B(\R T, \tau_B \R T)=0$ and $\Hom_B(P,\R T)=0$, then $\F \R T \subseteq {}^\perp(\tau_B\R T)\cap P^{ \perp}$.

Let  $Y\in {}^\perp(\tau_B\R T)\cap P^{ \perp}$. By Proposition \ref{prop_ortog}, we have that $\E Y \in {}^\perp (\tau_A M)$. On the other hand,
\begin{eqnarray*}
\Hom_A(P,\E Y)&=&\Hom_B(\R P, Y)\\
&=& \Hom_B(P,Y)\\
&=&0.
\end{eqnarray*}
Then, $\E Y \in  {}^\perp(\tau_A T)\cap P^{ \perp}= \F T$. Thus, $Y \cong \R\E Y \in \F \R T$. Hence $(\R M, P^*)$ is a support $\tau$-tilting pair for $\mo\,B$.

\end{proof}

In our next corollary, we state a particular case of the above result.

\begin{cor} Let $B$ an algebra and $A=B[P_0]$. Then,
\begin{enumerate}[(a)]
\item If $M$ is a $\tau$-tilting $B$-module,  then $\E M\oplus S$ is a  $\tau$-tilting $A$-module.
\item If $T$ is a $\tau$-tilting $A$-module. Then $\R T$ is  a $\tau$-tilting $B$-module.
\end{enumerate}
\end{cor}

It follows directly from Theorem \ref{teopares} that we get morphisms between the corresponding posets of support $\tau$-tilting modules, as we show in the following corollary.

\begin{cor}\label{cormaps}
The functors  $\E$ and $\R$ induce two maps:
\begin{eqnarray*}
e: \sti \, B &\rightarrow& \sti\, A\\
(M,Q) &\rightarrow&  (\E M \oplus S, Q)
\end{eqnarray*}
and,
\vspace{-0.3in}
\begin{eqnarray*}
r: \sti \,A &\rightarrow& \sti\, B\\
(T,P) &\rightarrow& (\widehat{T},P^*)
\end{eqnarray*}
where  $\widehat{T}$ is a (unique up to isomorphism) basic $\tau$-rigid $B$-module such that $\mbox{add}\widehat{T} = \mbox{add}\R T$.  Moreover, the composition $re=\mbox{id}_{ \sti \, B}$.
\end{cor}

\begin{proof}
By Theorem (\ref{teopares}), $r$ and $e$ are maps. Moreover, the relation $re=\mbox{id}_{ \sti \, B}$ follows from $\R\E \cong  \mbox{id}_{\mo\,B}$.
\end{proof}

Now, we discuss the torsion pairs corresponding to a $\tau$-tilting module $T$. We recall that if $T$ is a $\tau$-tilting module over an algebra $C$, then $T$ determines a torsion pair $(^{\bot}\tau T, T^{\bot})$ in $\mo\,C$.

\begin{teo}
\begin{enumerate}[(i)]
\item Let $T$ be a $\tau$-tilting $B$-module and $X$ be a $B$-module. Then the following conditions hold.
\begin{enumerate}[(a)]
\item $X \in ^{\bot}\tau_B T$ if and only if $\E X \in ^{\bot}(\tau_A \E T)$.
\item  $X \in T^{\bot}$ if and only if $\E X \in \E T^{\bot}$.
\end{enumerate}
\item Let $T$ be a $\tau$-tilting $A$-module. Then the following conditions hold.
\begin{enumerate}[(a)]
\item If $(^{\bot}\tau_A T, T^{\bot})$ is a hereditary torsion pair for $\mo\,A$ then $(^{\bot}(\tau_B \R T), (\R T)^{\bot})$ is an hereditary torsion pair for $\mo\,B$.
\item If $(^{\bot}\tau_A T, T^{\bot})$ is a splitting torsion pair for $\mo\,A$ then $(^{\bot}(\tau_B \R T), (\R T)^{\bot})$ is a splitting torsion pair for $\mo\,B$.
\end{enumerate}
\end{enumerate}
\end{teo}

\begin{proof}
$(i).$
$(a).$ Since $T$ is a $\tau$-tilting $A$-module, we know that $^{\bot}\tau_A T= \F T$. Then the result follows from the fact that $X\in \F T$ if and only if $\E X\in \F \E T$.

$(b).$ Follows from the fact that
\begin{eqnarray*}
\Hom_A(\E T, \E X)&\cong& \Hom_B(\R \E T, X)\\
& \cong & \Hom_B(T,X).
\end{eqnarray*}

$(ii).$
$(a).$ Consider  $(^{\bot}\tau_A T, T^{\bot})$  a hereditary torsion pair for $\mo\,A$. Let $X\in ^{\bot}(\tau_B \R T)$ and $Y$ be a submodule of $X$. Then, we shall show that $Y \in ^{\bot}(\tau_B \R T)$.

Since $X\in ^{\bot}(\tau_B \R T)$, by Proposition \ref{prop_ortog}, we have that $\E X\in ^{\bot}\tau_A  T$.  Then $\E N \in ^{\bot}\tau_A  T$, because $\E N$ is a submodule of $\E M$. Since $^{\bot}\tau_A  T= \F T$, then $\E N \in  \F T$. Thus, $N \in \F \R T=^{\bot}(\tau_B \R T)$.  Therefore $(^{\bot}(\tau_B \R T), (\R T)^{\bot})$ is a hereditary torsion pair for $\mo\,B$.

$(b).$ Suppose  $(^{\bot}\tau_A T, T^{\bot})$ is a splitting torsion pair for $\mo\,A$ and consider $X\in \mo\,B$. Since $\E X \in \mo\,A$, we have that either  $\E X \in ^{\bot}\tau_A T = \F T$ or $\E X \in T^{\bot}$. Therefore, $X\in ^{\bot}(\tau_B \R T) $ or $X\in (\R T)^{\bot}$ and the assertion is shown.

\end{proof}

We end this section  computing the endomorphism algebra of $e T$, when $T$ is a $\tau$-tilting $B$-module. Recall that  $\nu_C=DC\otimes_C\_$  is the \emph{Nakayama functor} for an algebra $C$.

\begin{teo}
Let $T$ be a $\tau$-tilting $B$-module. Then, $\mbox{End}_A eT $ is the one-point extension of $\mbox{End}_BT$ by the module $\Hom_B(T,\nu_B P_0)$.
\end{teo}

\begin{proof}

Note that

\[\mbox{End}_A eT =\mbox{End}_A (\E T\oplus S)\cong\left(
\begin{array}{lr}
\mbox{End}_A(\E T) & \Hom_A(\E T,S)  \\
\Hom_A(S, \E T)& \mbox{End}_AS  \\
\end{array}
\right). \]

Since $\mbox{End}_AS \cong k$ and $\E T\in S^{\tiny{\mbox{perp}}}$, it is only left to prove that $\Hom_A(\E T, S)\cong \Hom_B(T,\nu_BP_0)$.

Consider the Auslander-Reiten sequence
\begin{equation}\label{ass}
0\rightarrow \tau_AS \rightarrow E \rightarrow S \rightarrow 0
\end{equation}
in $\mo\,A$. By \cite[IV,3.9]{ASS}, $E$ is an injective module. We claim that $\R E\cong \nu_B P_0$. Indeed, applying $\R$ to the sequence (\ref{ass}), we obtain $\R E \cong \R (\tau_A S)$.

On the other hand, consider the projective resolution of $S$,
\[ 0\rightarrow P_0 \rightarrow P \rightarrow S.\]
By \cite[IV,2.4]{ASS}, there exists an exact sequence

\begin{equation}\label{nakayamaseq}
0\rightarrow \tau_AS \nu_AP_0 \rightarrow \nu_AP \rightarrow \nu_A S \rightarrow 0
\end{equation}
where $\nu_AP\cong S$ and $\nu_AP_0=\bigoplus_xI_x^A$, if $P_0=\bigoplus_xP_x^A$ where $P_x^A$ is the indecomposable projective $A$-module at the vertex x. By \cite[Lemma 4.5]{AHT}, $I_x^A=\E I_x^B$. Then, applying the functor $\R$ to (\ref{nakayamaseq}) we obtain that $\R(\tau_AS)\cong \R(\nu_AP_0)\cong \nu_BP_0$. Therefore,
\begin{eqnarray*}
\R E &\cong  &\R (\tau_A S)\\
&\cong & \nu_BP_0.
\end{eqnarray*}
Applying $\Hom_A(\E T, \--)$ to the sequence (\ref{ass}) yields an exact sequence as follows
\[0\rightarrow \Hom_A(\E T,\tau_AS)\rightarrow \Hom_A(\E T, E)\rightarrow \Hom_A(\E T,S)\rightarrow  \Ex_A(\E T, \tau_AS).\]

Since $\mbox{pd}_AS \leq 1$, the Auslander-Reiten formula yields $ \Hom_A(\E T,\tau_AS)=0$. On the other hand, since $\Ex_A(\E T, \tau_AS)\cong D \overline{\Hom}_A(S,\E T)$ and $\Hom_A(S,\E T)=0$, we obtain that $\Ex_A(\E T, \tau_AS)=0$. Thus, $\Hom_A(\E T, E)\cong \Hom_A(\E T,S)$.

Finally, since $E \in S^{\tiny{\mbox{perp}}}$, then
\begin{eqnarray*}
 \Hom_A(\E T, S)&\cong & \Hom_A(\E T, E)\\
 & \cong & \Hom_A( T,\R  E)\\
 & \cong & \Hom_B(T,\nu_BP_0)
 \end{eqnarray*}
proving the result.

\end{proof}

\section{The quiver of support $\tau$-tilting modules}

Now we focus our attention on the quivers of the support $\tau$-tilting modules. We shall compare $Q(\sti\,B)$ and $Q(\sti\,A)$. We aim is to show that the morphism $e$ states in corollary \ref{cormaps} is a full embedding between the posets of support $\tau$-tilting modules. We start with the following Theorem.

\begin{teo}\label{teoembe}
\begin{enumerate}[(a)]
\item The maps  $e: \sti \, B \rightarrow \sti\, A$ and $r: \sti \,A \rightarrow \sti\, B$ are morphisms of posets.
\item An arrow  $\alpha: (M_1,Q_1)\rightarrow (M_2,Q_2)$ in  $Q(\sti\,B)$ induces an arrow
\newline $e\alpha: e(M_1,Q_1)\rightarrow e(M_2,Q_2)$ in $Q(\sti \,A)$.
\end{enumerate}
\end{teo}

\begin{proof}

$(a).$ Let $(M_1,Q_1)$ and $(M_2,Q_2)$ be support $\tau$-tilting pairs for $\mo\,B$ such that $(M_1,Q_1)<(M_2,Q_2)$. We have to prove that $(\E M_1\oplus S,Q_1)<(\E M_2 \oplus S, Q_2)$,  or equivalently, $\F (\E M_1 \oplus S)\subseteq \F(\E M_2 \oplus S)$. Since $\F (\E M_1 \oplus S)=\F (\E M_1)\oplus \F S$, we only have to show that $\F (\E M_1)\subseteq \F (\E M_2)$.

Since $\F M_1 \subseteq \F M_2$,  there exists an epimorphism $f: Z\rightarrow M_1$, with $Z \in \ad M_2$. Applying the exact functor $\E$ to $f$, we obtain an epimorphism $\E f:\E Z \rightarrow \E M_1$, where $\E Z \in \ad \E M_2$. Then, $\E M_1 \in \F (\E M_2)$. Therefore, $\F (\E M_1)\subseteq \F (\E M_2)$.

Conversely, let $(T_1,P_1)$ and $(T_2,P_2)$  be support $\tau$-tilting pairs for $\mo\,A$, such that $(T_1,P_1)<(T_2,P_2)$. We claim that $\R T_1 \in \F \R T_2$. In fact, since $\F T_1 \subseteq \F T_2$, there exists an epimorphism $g: W \rightarrow T_1$, with $W \in \ad T_2$. Applying the exact functor  $\R$ to $g$, we obtain an epimorphism $\R g: \R W\rightarrow \R T_2$, where $\R W \in \ad \R T_2$. Therefore, $\R T_1 \in \F(\R T_2)$.

$(b).$ Let $\alpha: (M_1,Q_1)\rightarrow (M_2,Q_2)$ be an arrow in $Q(\sti \, B)$. Then, there exists an almost complete support $\tau$-tilting pair for $\mo\,B$, let denote it $(U,P)$, which is a direct summand of $(M_1,Q_1)$ and $ (M_2,Q_2)$. Since $e$ is a morphism of posets, we have $e(M_1,Q_1)<e(M_2,Q_2)$. Observe that $e(U,P)=(\E U\oplus S,P)$ is an almost complete support  $\tau$-tilting pair for $\mo\,A$, since
\begin{eqnarray*}
|\E U \oplus S|+|Q|&=&|\E U|+1+|Q|\\
&=&|U|+|Q|+1\\
&=&n-1.
\end{eqnarray*}
Moreover,  $e(U,P)=(\E U\oplus S,P)$ is a direct summand of $e(M_1,Q_1)$ and $e(M_2,Q_2)$. Thus, by definition, we have that
$e(M_2,Q_2)= \mu_{\E X}^{-}e(M_1,Q_1)$. Hence, there exists an arrow $e\alpha: e(M_1,Q_1)\rightarrow e(M_2,Q_2)$ in $Q(\sti \,A)$.

\end{proof}

\begin{obs}
The above theorem  shows  that the Extension functor behaves well respect to the mutation of support $\tau$-tilting modules. In some way, the Extension functor commutes with the mutation.
\end{obs}

\begin{proof}[Proof of Theorem B]
By Theorem \ref{teoembe} and since $re=\mbox{Id}_{\tiny{\sti\,B}}$, the map $e$ is an embedding of quivers. Hence, we only have to show that if there exists an arrow $e(M,P)\rightarrow e(N,Q)$ in $Q(\sti\,A)$, then there exist an arrow $(M,P)\rightarrow (N,Q)$ in $Q(\sti\,B)$.

We know that $e(M,P)=(\E M\oplus S, P)$ and $e(N,Q)=(\E N \oplus S, Q)$. Since there exist an arrow from $e(M,P)$ to $e(N,Q)$, then there is an almost complete support $\tau$-tilting module, $(U,L)$, which is a direct summand of $e(M,P)$ and $e(N,Q)$. Since $S$ is a direct summand of both, then  $S$ is a direct summand of $U$.  Thus $U=U'\oplus S$, with $U'\in S^{\tiny{\mbox{perp}}}$ . Then, $|\R U|+|L|=|\R U'|+|L|=|U'|+|L|=n-2$.  Note that $L$ is a projective $B$-module, since $\Hom_A(L,S)=0$. Therefore, we have that  $(U',L)$ is an almost complete support $\tau$-tilting pair for $\mo\,B$ which is a direct summand of $(M,P)$ and $(N,Q)$. Since $r$ is a morphism of posets, there exists an arrow $(M,P)\rightarrow (N,Q)$ in $Q(\sti\,B)$.
\end{proof}

We illustrate the above theorem  with the following example.

\begin{ej}
Let $B$ be  the algebra given by the quiver
$ \xymatrix{
1 \ar@/_/[r]_{\beta}
& 2 \ar@/_/[l]_{\alpha}
}$
with the relation $\alpha\beta=0$.


We denote all the modules by their composition factors.
Consider $A=B[P_2]$, the one-point of $B$ by the projective $P_2={\small\txt{2\\1}}$. Then $A$ is given by the quiver $ \xymatrix{
1 \ar@/_/[r]_{\beta}
& 2 \ar@/_/[l]_{\alpha}
& 3 \ar[l]_{\gamma}
}$
with relation the $\alpha\beta=0$.


The quiver $Q(\sti\,A)$ is the following

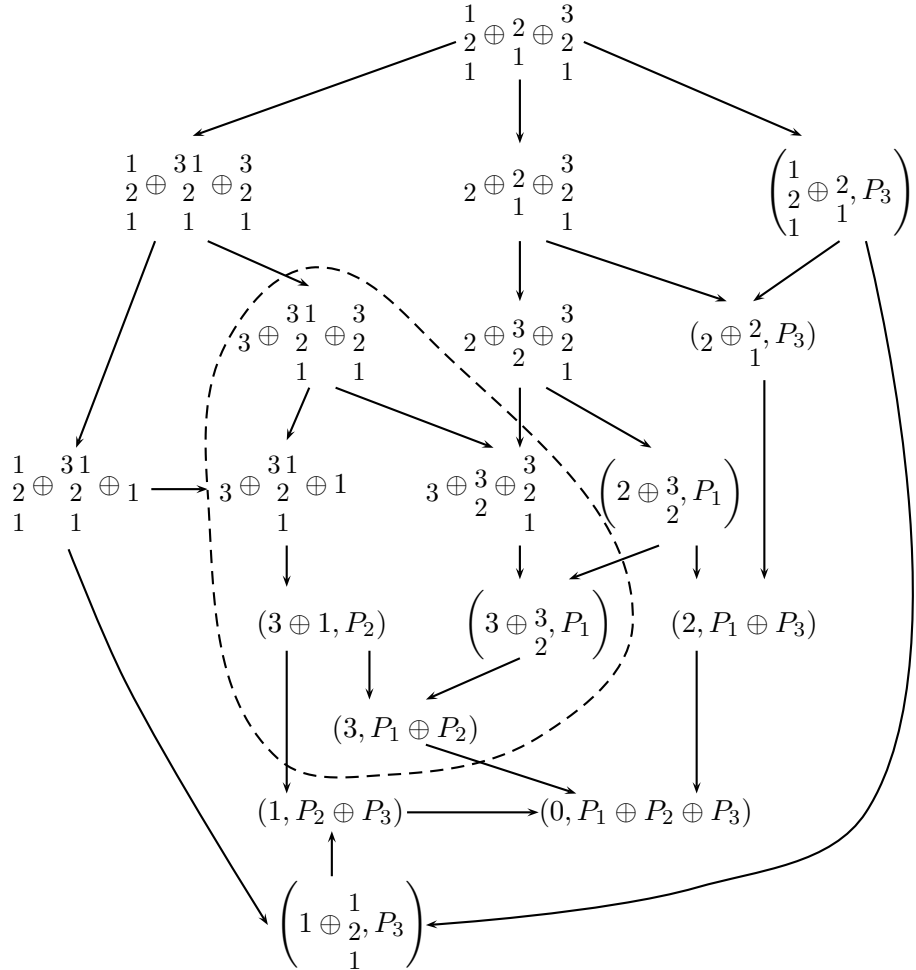
\begin{center}
\psset{unit=5mm,doublesep=5pt}
\begin{pspicture}(-2,-13)(4,14)

\uput{0}[0](0,0){${\small\txt{3}}\oplus{\small\txt{3\\2}}\oplus{\small\txt{3\\2\\1}}$}
\uput{0}[0](4.5,0){${\small\left(2\oplus \txt{3\\2},P_1\right)}$}
\uput{0}[0](-5.5,0){${\small\txt{3}}\oplus{\small\txt{3\,1\\2\\1}}\oplus{\small1}$}
\uput{0}[0](-11,0){${\small\txt{1\\2\\1}}\oplus{\small\txt{3\,1\\2\\1}}\oplus{\small\txt{1}}$}

\uput{0}[0](1,4){${\small\txt{2}}\oplus{\small\txt{3\\2}}\oplus{\small\txt{3\\2\\1}}$}
\uput{0}[0](7,4){$({\small\txt{2}}\oplus{\small\txt{2\\1}},{\small P_3})$}
\uput{0}[0](-5,4){${\small\txt{3}}\oplus{\small\txt{3\,1\\2\\1}}\oplus{\small\txt{3\\2\\1}}$}

\uput{0}[0](-8,8){${\small\txt{1\\2\\1}}\oplus{\small\txt{3\,1\\2\\1}}\oplus{\small\txt{3\\2\\1}}$}
\uput{0}[0](1,8){${\small\txt{2}}\oplus{\small\txt{2\\1}}\oplus{\small\txt{3\\2\\1}}$}
\uput{0}[0](9,8){$\left({\small\txt{1\\2\\1}}\oplus{\small\txt{2\\1}},P_3\right)$}
\uput{0}[0](1,12){${\small\txt{1\\2\\1}}\oplus{\small\txt{2\\1}}\oplus{\small\txt{3\\2\\1}}$}
\uput{0}[0](1,-3.5){${\small\left(3\oplus \txt{3\\2},P_1\right)}$}
\uput{0}[0](6.5,-3.5){${\small\left(2,P_1\oplus P_3\right)}$}
\uput{0}[0](-4.5,-3.5){${\small\left(3\oplus1,P_2\right)}$}
\uput{0}[0](-2.5,-6.3){${\small\left(3,P_1\oplus P_2\right)}$}
\uput{0}[0](-4.5,-8.5){${\small\left(1,P_2\oplus P_3\right)}$}
\uput{0}[0](3,-8.5){${\small\left(0,P_1\oplus P_2\oplus P_3\right)}$}
\uput{0}[0](-4,-11.5){${\small\left(1\oplus{\small\txt{1\\2\\1}},P_3\right)}$}

\psline{->}(0.8,12)(-6.2,9.5)
\psline{->}(2.5,11)(2.5,9.3)
\psline{->}(4.2,12)(10,9.5)
\psline{->}(-5.8,6.7)(-3.1,5.5)
\psline{->}(2.5,6.7)(2.5,5.1)
\psline{->}(3.2,6.7)(8,5.1)
\psline{->}(11,6.7)(8.7,5.1)
\psline{->}(-7.2,6.7)(-9.3,1.2)
\psline{->}(3.2,2.8)(6,1.2)
\psline{->}(2.5,2.8)(2.5,1.2)

\psline{->}(-3.1,2.8)(-3.7,1.4)
\psline{->}(-2.4,2.8)(1.8,1.2)
\psline{->}(-3.7,-1.4)(-3.7,-2.5)
\psline{->}(2.5,-1.4)(2.5,-2.3)
\psline{->}(6.2,-1.4)(3.8,-2.3)
\psline{->}(7.2,-1.4)(7.2,-2.3)
\psline{->}(9,3)(9,-2.3)
\psline{->}(-7.3,0.1)(-5.8,0.1)
\psline{->}(-1.5,-4.2)(-1.5,-5.5)
\psline{->}(2.5,-4.4)(0,-5.5)
\psline{->}(-3.7,-4.2)(-3.7,-8)
\psline{->}(-0.5,-8.5)(3,-8.5)
\psline{->}(0,-6.7)(4,-8)
\psline{->}(7.2,-4.2)(7.2,-8)
\psline{->}(-2.5,-10.2)(-2.5,-9)
\pscurve[arrows=->](-9.5,-1.5)(-8,-5)(-4.2,-11.5)
\pscurve[arrows=->](11.7,6.7)(11.7,-8.5)(7.2,-10.5)(0,-11.5)

\psccurve[linestyle=dashed](-5.3,4)(-5.8,0)(-3.9,-7.1)(-1,-7.5)(3,-6.5)(5.5,-3.5)(0.4,3.8)(-3,6)

\end{pspicture}
\end{center}

Then, the image of the quiver $Q(\sti\,B)$  under $e$ is the subquiver indicated by dotted lines.
\end{ej}

For the remainder of this section, we state some technical results about the local behaviour of $Q(\sti\,A)$. We are interested to know when the image of $e$ is closed by successors. The next theorem gives us an answer for a particular case.

\begin{teo}\label{teoimagen}
Let $(T,P)$ and $(T',P')$ be  basic  support $\tau$-tilting pairs for $\mo\,A$ such that there exists an arrow  $(T,P)\rightarrow (T', P')$ in $Q(\sti\,A)$.
If $(T,P)=e(M,Q)$ and $\Hom_A(\E M, S)\neq 0$ then there exists a  support $\tau$-tilting pair $(N,R)$ in $\sti \,B$ such that $(T', P')=e(N,R)$.

\end{teo}

\begin{proof}

Let $(T,P)=e(M,Q)$ be a support $\tau$-tilting pair for $\mo\,A$ such that $\Hom_A(\E M, S)\neq 0$. Then, by Shur's Lemma $S\in \F (\E M)$. We claim that $S$ is a direct summand of $T'$ where $(T',P')$ is a support $\tau$-tilting pair such that there exists an arrow from $(T,P)$ to $(T',P')$ in $Q(\sti\,A)$. In fact, otherwise $(T',P')=\mu_S(T,P)$. Moreover, since there exists an arrow from $(T,P)$ to $(T',P')$ in $Q(\sti \,A)$ then $(T',P')=\mu_S^{-}(T,P)$. Therefore, it follows by Definition \ref{leftmutation} that $S \notin \F (\E M)$, which is a contradiction. Hence, $T'=S\oplus Y$.

Since $S\oplus Y$ is a basic $\tau$-rigid module, then $\mbox{Ext}_A^1(S,Y)=0$ and $\Hom_A(S,Y)=0$. Then $Y\in S^{\tiny{perp}}$ and therefore $Y \cong \E \R Y$. Furthermore, since $\Hom_A(P',S\oplus Y)=0$ we have that $P'$ is a projective $B$-module. Considering  the support $\tau$-tilting pair $(\R Y, P')$ we obtain the result.


\end{proof}

The following example shows that the condition $\Hom_A(\E M,S)\neq 0$ in Theorem \ref{teoimagen} can not be remove.

\begin{ej}

Consider the following algebras:

\begin{eqnarray*}
B:
{
    \xymatrix  @!0 @R=0.7cm  @C=0.8cm {
       1& &     \\
       &3\ar[ul]^{\beta} \ar[dl]^{\gamma} &4 \ar[l]^{\alpha}   \\
      2 & &  }
} && \hspace{10mm}
B[P_3]:
{
    \xymatrix  @!0 @R=0.7cm  @C=0.8cm {
       1& &  4\ar[dl]_{\alpha}   \\
       &3\ar[ul]_{\beta} \ar[dl]^{\gamma} &   \\
      2 & & 5\ar[ul]^{\delta} }
} \\
 \alpha \beta =0 && \hspace{35mm}\alpha \beta =0
\end{eqnarray*}

It is not hard to see that $(1 \oplus 4, P_2\oplus P_3)$ is an almost complete support $\tau$-tilting pair for $\mo\,A$ and their complements are $(5 ,0)$ and $ (0,P_5)$. Moreover, there exists an arrow $(1 \oplus 5 \oplus 4, P_2\oplus P_3)\rightarrow (1\oplus 4,P_2 \oplus P_3\oplus P_5)$ in $Q(\sti\,A)$.

Note that a support $\tau$-tilting pair, $(U,P)$,  belongs to the image of $e$ if and only if $S$ is a direct summand of $U$. Then,  $(1 \oplus 5 \oplus 4, P_2\oplus P_3)$ belongs to the image of $e$, but  $(1\oplus 4,P_2 \oplus P_3\oplus P_5)$ does not belong to the image of $e$.
\end{ej}

Suppose we have a pair $(M,Q)$ in $Q(\sti\,A)$ which belongs to the image of $e$. Then, the following result gives information about the predecessors of $(M,Q)$.

\begin{teo}\label{teopredecesores}
Let $(T,P)$ be a support $\tau$-tilting pair such that there exists a support $\tau$-tilting pair $(M,Q)$ in $\sti\,B$ with $(T,P)=e(M,Q)=(\E M \oplus S, P)$. Then there is exactly one immediate predecessor of $(T,P)$ in $Q(\sti\,A)$ which does not belong to the image of $e$ if and only if $\Hom_A(\E M, S)\neq 0$.
\end{teo}

\begin{proof}

Suppose that there is exactly one immediate predecessor of $(T,P)$ in $Q(\sti\,A)$ which does not belong to the image of $e$ and assume $\Hom(\E M, S)=0$. Then, $S \notin \F (\E M)$. By definition $\mu_S(T,P)$ is a left mutation of $(T,P)$ and there exists an arrow from $(T,P)$ to $\mu_S(T,P)$ in $Q(\sti \,  A)$. Therefore, all the predecessors $(T',P')$ of $(T,P)$ satisfy that $T'=S\oplus M$ with $M \in S^{\tiny{perp}}$. Then, all the predecessors belong to the image of $e$, which is a contradiction.

Conversely, let  $(T,P)\in \sti\,A$ such that  $(T,P)=(\E M \oplus S, P)$ and $\Hom_A(\E M, S)\neq 0$. We  show that  there is only one immediate predecessor of $(T,P)$ which does not have $S$ as a direct summand.

By definition of $Q(\sti\,A)$,  there is at most  one immediate predecessor of  $(T,P)$ such that $S$ is not a  direct summand.  Assume that all immediate predecessors of $(T,P)$ have the simple $S$ as a direct summand. Then there exists an immediate successor  of $(T,P)$, let say $(T',P')$ in $Q(\sti\,A)$, such that $S$ is not a direct summand of $(T',P')$. Thus, by construction, we have $(T',P')=\mu_{S}^{+}(T,P)$. It follows by Definition \ref{leftmutation} that $S \notin \F (\E M)$ and thus $\Hom_A(\E M, S)=0$, which is a contradiction. Therefore, we prove that there is exactly one immediate predecessor of $(T,P)$ such that $S$ is not a direct summand.

Conversely,
\end{proof}








We end up this section showing an example that if we extend by a non-projective module, then neither the restriction nor the extension define maps between the corresponding posets of support $\tau$-tilting modules.

\begin{ej}
Let $B$ the following algebra

$$
   \xymatrix  @!0 @R=0.7cm  @C=0.8cm {
       1& &     \\
       &3\ar[ul]^{\beta} \ar[dl]^{\gamma} &4 \ar[l]^{\alpha}   \\
      2 & &  }
$$

and let $A=B[X]$, where $X=\small{\txt{3\\2}}$. Then $A$ is given by the quiver

$$\xymatrix  @!0 @R=0.7cm  @C=0.8cm {
       1& &  4\ar[dl]_{\alpha}   \\
       &3\ar[ul]_{\beta} \ar[dl]^{\gamma} &   \\
      2 & & 5\ar[ul]^{\delta} }$$

with the relation $\delta \beta =0$.

\begin{enumerate}[(a)]
\item Extending the $\tau$-tilting $B$-module $M=\small{\txt{4\\3\\2}\oplus \txt{3}\oplus \txt{4\\3}\oplus \txt{4\\3\\1}}$ we get the $A$-module $eM=\small{\txt{45\\3\\2}\oplus \txt{5\\3}\oplus \txt{45\\3}\oplus \txt{4\\3\\1}\oplus \txt{5}}$ which is not $\tau$-tilting because $\Hom_A(\small{\txt{4\\3\\1}},\tau_A\,5) \neq 0$.
\item Restricting the $\tau$-tilting $A$-module $T=4\oplus 5 \oplus \small{\txt{45\\3}\oplus\small{\txt{45\\3\\2}}\oplus1}$ yields the $B$-module $\R T= 4 \oplus \small{\txt{4\\3}\oplus\small{\txt{4\\3\\2}}\oplus1}$ which is not $\tau$-tilting because $\Hom_B(1,\tau_B\,\,\small{\txt{4\\3\\2}})\neq 0$.
\end{enumerate}

\end{ej}

\end{document}